\documentclass[11pt]{amsart}

\usepackage{amsmath}
\usepackage{amssymb}
\usepackage{bm} 
\usepackage{mathrsfs} 
\usepackage{tensor} 
\usepackage{braket} 
\usepackage{physics} 

\DeclareMathOperator{\Vol}{Vol}

\DeclareMathOperator{\Ric}{Ric}

\numberwithin{equation}{section}

\usepackage{hyperref}
\usepackage{amsthm} 

\usepackage{cleveref} 

\usepackage{autonum} 

\crefname{theorem}{Theorem}{Theorems}
\crefname{lemma}{Lemma}{Lemmas}
\crefname{proposition}{Proposition}{Propositions}
\crefname{corollary}{Corollary}{Corollaries}
\crefname{claim}{Claim}{Claims}
\crefname{assumption}{Assumption}{Assumptions}
\crefname{definition}{Definition}{Definitions}
\crefname{remark}{Remark}{Remarks}
\crefname{example}{Example}{Examples}
\crefname{equation}{}{}
\crefname{section}{Section}{Sections}
\crefname{subsection}{Section}{Sections}
\crefname{appendix}{Appendix}{Appendices}

\crefname{condition}{}{}

\creflabelformat{equation}{(#2#1#3)}
\creflabelformat{condition}{(#2#1#3)}

\theoremstyle{plain}
\newtheorem{theorem}{Theorem}[section]
\newtheorem{lemma}[theorem]{Lemma}

\newtheorem{assumption}[theorem]{Assumption}

\theoremstyle{definition}

\theoremstyle{remark}



\AtBeginDocument{%
	\def\MR#1{}
}

\begin{document}
\title[On the renormalized volume of tubes]{On the renormalized volume of tubes over polarized K\"{a}hler-Einstein manifolds}
\author{Yuya Takeuchi}
\address{Graduate School of Mathematical Sciences \\ The University of Tokyo \\ 3-8-1 Komaba, Meguro, Tokyo 153-8914 Japan}
\email{ytake@ms.u-tokyo.ac.jp}

\begin{abstract}
	A formula of the renormalized volume of tubes
	over polalized K\"{a}hler-Einstein manifolds
	is given in terms of the Einstein constant and the volume of the polarization.
\end{abstract}

\maketitle

\section{Introduction} \label{section:introduction}

The renormalized volume
of conformally compact Einstein manifolds
was first introduced by Henningson-Skenderis~\cite{Henningson-Skenderis98}
and Graham~\cite{Graham00}
in studies of the AdS/CFT correspondence.
Let $M$ be an $(m+1)$-dimensional compact manifold with boundary $\partial M$,
and $x$ a defining function of the boundary.
A Riemannian metric $g_{+}$ on the interior of $M$
is called a conformally compact Einstein metric
if $x^{2} g_{+}$ extends to a Riemannian metric on $M$
and $\Ric_{g_{+}} = - m g_{+}$ holds.
The conformal class $C = [(x^{2} g_{+})|_{\partial M}]$ on $\partial M$
is independent of the choice of $x$,
called the conformal infinity of $g_{+}$.
If we choose a representative $g$ of $C$,
there exists a unique defining function $x$ near the boundary
such that $|d x|_{g_{+}} =1$ and $(x^{2} g_{+})|_{\partial M} = g$.
Note that this $x$ changes with a representative $g \in C$.
For such a defining function,
consider the relatively compact domain $M^{\epsilon} = \{ x < - \epsilon \}$ for $\epsilon > 0$.
The volume $\Vol_{g_{+}}(M^{\epsilon})$ of $M^{\epsilon}$
has the following expansion, as $\epsilon \to + 0$:
\begin{equation}
	\Vol_{g_{+}}(M^{\epsilon})
	= \sum_{j = 0}^{\lceil m/2 \rceil -1} a_{j} \, \epsilon^{2 j - m}
	+
	\begin{cases}
		V^{o} + O(1), & m \ \text{odd}, \\
		L \log \epsilon + V^{e} + O(1), & m \ \text{even}.
	\end{cases}
\end{equation}
The constant term $V^{o}$ or $V^{e}$ in this asymptotics is called the renormalized volume.
Graham~\cite{Graham00}
has proved that $V^{e}$ and $L$ are independent of the choice of $g$.
Moreover,
Graham and Zworski~\cite{Graham-Zworski03} have shown that $L$
coincides with a constant multiple of the total $Q$-curvature of the boundary,
a global conformal invariant;
see also~\cite{Fefferman-Graham02}.
Fefferman and Graham~\cite{Fefferman-Graham02}
also have deduced a formula for $V^{o}$ in terms of the scattering matrix for $(M, g_{+})$.
A similar formula for $V^{e}$ has been proved by
Yang-Chang-Qing~\cite{Yang-Chang-Qing08}.

Now we consider the renormalized volume of strictly pseudoconvex domains.
Let $\Omega$ be a bounded strictly pseudoconvex domain
in an $(n+1)$-dimensional complex manifold $X$.
Assume that
$\Omega$ has a complete K\"{a}hler-Einstein metric $\omega_{+}$ with Einstein constant $- (n+2)$.
We also suppose that $K_{X}$ has a Hermitian metric
that is flat on the pseudoconvex side near the boundary.
Then
there exists a defining function $\rho$
such that $\omega_{+}$ coincides with $- d d^{c} \log (- \rho)$ near the boundary,
where $d^{c} = (\sqrt{-1}/2) (\overline{\partial} - \partial)$;
see \cref{lem:existence-of-potential}.
Note that such a defining function $\rho$ is not uniquely determined by $\omega_{+}$.
For this $\rho$,
the volume $\Vol_{\omega_{+}}(\Omega^{\epsilon})$
of the relatively compact domain $\Omega^{\epsilon} = \{ \rho < - \epsilon \}$
has an expansion of the form
\begin{equation}
	\Vol_{\omega_{+}}(\Omega^{\epsilon})
	= \sum_{j = 0}^{n} b_{j} \, \epsilon^{j - n - 1} + V + O(1),
\end{equation}
as $\epsilon \to + 0$.
We call the constant term $V$ the \emph{renormalized volume} of $(\Omega, \omega_{+})$.
Note that Seshadri~\cite{Seshadri07} has considered
the renormalized volume for the choice of defining functions similar to the conformal case.

In this paper,
we compute the renormalized volume
of tubes over polarized K\"{a}hler-Einstein manifolds.
Let $Y$ be a closed K\"{a}hler manifold
and $L$ an ample line bundle over $Y$.
Then the pair $(Y, L)$ is called a \emph{polarized K\"{a}hler manifold}.
A \emph{polarized K\"{a}hler-Einstein manifold with Einstein constant $\beta$}
is a triple $(Y, L, h_{L})$ consisting of a polarized K\"{a}hler manifold $(Y, L)$
and a Hermitian metric $h_{L}$ of $L$
such that the curvature form $\omega = \sqrt{-1} \Theta_{h_{L}}$
defines a K\"{a}hler-Einstein metric on $Y$ with Einstein constant $\beta$.

\begin{theorem} \label{thm:renormalized-volume-for-tubes}
	Let $(Y, L, h_{L})$ be an $n$-dimensional polarized K\"{a}hler-Einstein manifold with Einstein constant $\beta$.
	If $\beta < 1$,
	there exists a complete K\"{a}hler-Einstein metric $\omega_{+}$
	on $\Omega = \{ v \in L^{-1} \mid h_{L^{-1}}(v, v) < 1\}$ with Einstein constant $- (n+2)$,
	where $h_{L^{-1}}$ is the dual metric of $h_{L}$,
	and the renormalized volume $V$ of $(\Omega, \omega_{+})$ is given by
	\begin{equation}
		V = (2 \pi)^{n+1} \left[ \frac{1}{(n+1)!} \left(- \frac{\beta}{n+1} \right)^{n+1}
		- \left(\frac{1 - \beta}{n+2} \right)^{n+1} \right] \Vol(L).
	\end{equation}
	Here $\Vol(L) = \int_{Y} c_{1}(L)^{n}$ is the volume of $L$.
\end{theorem}

Note that the assumption $\beta < 1$
is a necessary and sufficient condition
for the existence of a complete K\"{a}hler-Einstein metric on $\Omega$
with negative Einstein constant,
which has been proved by van Coevering~\cite{vanCoevering12}.

To prove \cref{thm:renormalized-volume-for-tubes},
we apply the formula of the renormalized volume
obtained by Hirachi-Marugame-Matsumoto~\cite{Hirachi-Marugame-Matsumoto17}.
Since $K_{X}$ has a Hermitian metric that is flat on the pseudoconvex side,
there exists a lift $\widetilde{c}_{1}(K_{\Omega}) \in H^{2}_{c}(\Omega; \mathbb{R})$
of the first Chern class $c_{1}(K_{\Omega}) \in H^{2}(\Omega; \mathbb{R})$.

\begin{theorem}[{\cite[Theorem 1.1]{Hirachi-Marugame-Matsumoto17}}]
\label{thm:decomposition-of-renormalized-volume}
	The renormalized volume $V$ of $(\Omega, \omega_{+})$ is given by
	\begin{equation}
		V
		= \frac{(-1)^{n+1}}{2 (n!)^{2} (n+1)!} \overline{Q}'(\partial \Omega)
		+ \left(\frac{2 \pi}{n+2} \right)^{n+1} \int_{\Omega} \widetilde{c}_{1}(K_{\Omega})^{n+1}.
	\end{equation}
	Here $\overline{Q}'(\partial \Omega)$ is
	the \emph{total $Q$-prime curvature} of $\partial \Omega$;
	see~\cite[Proposition 5.5]{Hirachi14} for the definition.
\end{theorem}

In the setting of \cref{thm:renormalized-volume-for-tubes},
the boundary $\partial \Omega$
is a Sasakian $\eta$-Einstein manifold,
and the total $Q$-prime curvature for such CR manifolds
has been computed by the author~\cite{Takeuchi} and Case-Gover~\cite{Case-Gover17}.
The result is
\begin{equation}
	\overline{Q}'(\partial \Omega)
	= 2 (n !)^{2} \left(\frac{2 \pi \beta}{n+1} \right)^{n+1} \Vol(L).
\end{equation}
Hence
it is enough to compute the integral of $\widetilde{c}_{1}(K_{\Omega})^{n+1}$.
Here,
we will give a little bit more general statement.
We consider a general bounded strictly pseudoconvex domain $\Omega$,
and make the following assumption:

\begin{assumption} \label{assumption:flat-Hermitian-metric}
	The maximal compact analytic set of $\Omega$
	is a smooth complex hypersurface $D$,
	and there exists a Hermitian metric $h_{\Omega}$ of $K_{\Omega}$
	such that the support of its curvature
	is contained in a tubular neighborhood of $D$.
\end{assumption}

Namely,
we assume that
the cohomology class $\widetilde{c}_{1}(K_{\Omega})$ is localized along a smooth divisor.
Then the quantity $\int_{\Omega} \widetilde{c}_{1}(K_{\Omega})^{n+1}$
can be computed as follows.

\begin{theorem} \label{thm:computation-of-Chern-numbers}
	Under \cref{assumption:flat-Hermitian-metric},
	the normal bundle $N_{D / \Omega}$ of $D$ satisfies
	$c_{1}(K_{D}) = \beta \cdot c_{1}(N_{D / \Omega})$
	in $H^{2}(D; \mathbb{R})$ for some $\beta \in \mathbb{R}$,
	and the quantity $\int_{\Omega} \widetilde{c}_{1}(K_{\Omega})^{n+1}$ is given by
	\begin{equation}
		\int_{\Omega} \widetilde{c}_{1}(K_{\Omega})^{n+1}
		= (\beta - 1)^{n+1} \int_{D} c_{1}(N_{D / \Omega})^{n},
	\end{equation}
	Additionally,
	assume that $\Omega$ has a K\"{a}hler-Einstein metric with negative Einstein constant.
	Then $N_{D / \Omega}$ is negative and $\beta < 1$ holds.
	In particular,
	$\int_{\Omega} \widetilde{c}_{1}(K_{\Omega})^{n+1}$ is negative.
\end{theorem}

Note that \cref{thm:computation-of-Chern-numbers} can be generalized to the case
that the maximal compact analytic set consists of disjoint smooth complex hypersurfaces.
We also remark that
the constant $\beta$ is a rational number
since $c_{1}(K_{D})$ and $c_{1}(N_{D / \Omega})$ are elements of $H^{2}(D; \mathbb{Z})$.
In particular,
$\int_{\Omega} \widetilde{c}_{1}(K_{\Omega})^{n+1}$ is also a rational number.

This paper is organized as follows.
In Section 2,
we recall fundamental facts for strictly pseudoconvex domains.
Section 3 provides the proof of \cref{thm:computation-of-Chern-numbers}.
In Section 4,
we prove \cref{thm:renormalized-volume-for-tubes}.

\section*{Acknowledgements}

The author is grateful to his supervisor Professor Kengo Hirachi
for various helpful comments.
He also thank Professor Paul Yang for introducing him to this work.
A part of this work was carried out during his visit to Princeton University
with the support from The University of Tokyo/Princeton University
Strategic Partnership Teaching and Research Collaboration Grant,
and from the Program for Leading Graduate Schools, MEXT, Japan.
This work was also supported by JSPS Research Fellowship for Young Scientists
and KAKENHI Grant Number 16J04653.

\section{Strictly pseudoconvex domains}

Let $\Omega$ be a bounded strictly pseudoconvex domain
in an $(n+1)$-dimensional complex manifold $X$.
If $X$ is the complex Euclidean space,
$\Omega$ has a complete K\"{a}hler-Einstein metric
with negative Einstein constant~\cite[Corollary 4.5]{Cheng-Yau80}.
For a general $X$,
a necessary and sufficient condition
for the existence of such a metric is obtained by van Coevering.

\begin{theorem}[{\cite[Theorem 3.1 and Proposition 3.2]{vanCoevering12}}]
\label{thm:existence-of-KE-metric}
	The domain $\Omega$ has a complete K\"{a}hler-Einstein metric $\omega_{+}$
	with Einstein constant $- (n + 2)$
	if and only if $K_{\Omega}$ has a Hermitian metric with positive curvature.
	Moreover,
	such a K\"{a}hler-Einstein metric is unique and invariant under biholomorphisms.
\end{theorem}

For $X = \mathbb{C}^{n+1}$,
this metric has a global K\"{a}hler potential of the form $- \log (- \rho)$,
where $\rho$ is a defining function of $\Omega$.
This fact does not hold
if, for example, $\Omega$ contains a compact analytic set with positive dimension.
However,
there exists such a K\"{a}hler potential near the boundary
under an assumption for $K_{\Omega}$.

\begin{lemma} \label{lem:existence-of-potential}
	Assume that $\Omega$ has a complete K\"{a}hler-Einstein metric $\omega_{+}$
	with Einstein constant $- (n + 2)$.
	There exists a defining function $\rho$ of $\Omega$
	such that $ \omega_{+} + d d^{c} \log (- \rho) $ has compact support in $\Omega$
	if and only if $K_{X}$ has a Hermitian metric $h_{X}$
	that is flat on the pseudoconvex side near the boundary.
\end{lemma}

\begin{proof}
	Let $h_{X}'$ be a Hermitian metric of $K_{X}$.
	From the proof of~\cite[Theorem 3.1]{vanCoevering12},
	there exists a defining function $\rho'$ of $\Omega$
	such that
	\begin{equation}
		\omega_{+} + d d^{c} \log (- \rho') = \sqrt{-1} (n+2)^{-1} \Theta_{h_{X}'}.
	\end{equation}
	This expression proves the equivalence.
\end{proof}

Note that this condition is equivalent to
the existence of a pseudo-Einstein contact form on the boundary%
~\cite[Proposition 2.6]{Hirachi-Marugame-Matsumoto17}.
In particular,
the compactly supported $(1, 1)$-form $(\sqrt{-1} / 2 \pi) \Theta_{h_{X}}$ defines a cohomology class
$\widetilde{c}_{1} (K_{\Omega})$ in $H^{2}_{c}(\Omega; \mathbb{R})$,
which is a lift of the first Chern class $c_{1}(K_{\Omega}) \in H^{2}(\Omega; \mathbb{R})$.
Note that the cohomology class
$\widetilde{c}_{1}(K_{\Omega})^{n+1} \in H^{2n+2}_{c}(\Omega; \mathbb{R})$
is independent of the choice of $h_{X}$.
If $X = \mathbb{C}^{n+1}$,
then $\widetilde{c}_{1}(K_{\Omega}) = 0$ in $H^{2}_{c}(\Omega; \mathbb{R})$
since $X$ has a flat K\"{a}hler metric.
However,
as in the statement of \cref{thm:computation-of-Chern-numbers},
$\int_{\Omega} \widetilde{c}_{1}(K_{\Omega})^{n+1}$ may take a negative value.


Before the end of this section,
we give a lemma for the normal bundle of the maximal compact analytic set.

\begin{lemma} \label{lem:normal-bundle-of-exceptional-set}
	Let $D$ be a smooth complex hypersurface in $\Omega$
	that is a connected component of the maximal compact analytic set of $\Omega$.
	Then the normal bundle $N_{D / \Omega}$ cannot be positive.
\end{lemma}

\begin{proof}
	Suppose that $N_{D / \Omega}$ is positive.
	Then as in the proof of \cite[Satz 1]{Schneider73},
	there exist a small neighborhood $U$ of $D$
	and a smooth strictly plurisubharmonic function $\psi$ on $U \setminus D$
	such that $\psi^{-1}((c, \infty)) \cup D$
	is an open neighborhood of $D$ for any $c \in \mathbb{R}$.
	On the other hand,
	there exists an analytic space $\Omega'$
	and a proper surjective holomorphic map $\phi \colon \Omega \to \Omega'$
	such that
	\begin{enumerate}
		\item $\phi$ maps $D$ to a single point $q \in \Omega'$;
		\item $\phi \colon \Omega \setminus D \to \Omega' \setminus \{ q \}$ is a biholomorphism;
		\item $\phi_{*} (\mathcal{O}_{\Omega}) = \mathcal{O}_{\Omega'}$.
	\end{enumerate}
	Let $U' = \phi(U)$ be an open set in $\Omega'$.
	Then $\psi' = \psi \circ \phi^{-1}$ is a plurisubharmonic function on $U' \setminus \{q\}$
	and $\psi'(q')$ goes to $\infty$ as $q' \to q$.
	However,
	from~\cite[Satz 4]{Grauert-Remmert56},
	$\psi'$ extends to a plurisubharmonic function on $U'$;
	this is a contradiction.
\end{proof}

\section{Proof of \texorpdfstring{\cref{thm:computation-of-Chern-numbers}}{Theorem 1.2}}

We first localize the integral $\int_{\Omega} \widetilde{c}_{1}(K_{\Omega})^{n+1}$ along $D$.
From \cref{assumption:flat-Hermitian-metric},
there exists a Hermitian metric $h$ of the normal bundle $N_{D / \Omega}$
such that the domain
\begin{equation}
	V = \{ v \in N_{D / \Omega} \mid h (v, v) < 1 \}
\end{equation}
is diffeomorphic (not necessarily biholomorphic)
to a tubular neighborhood of $D$ containing the support of the curvature of $h_{\Omega}$;
in the following
we identify the zero section of $N_{D / \Omega}$ with $D$.
Then
\begin{equation}
		\int_{\Omega} \widetilde{c}_{1} (K_{\Omega})^{n+1}
		= \int_{\Omega} \left(\frac{\sqrt{-1}}{2 \pi} \Theta_{h_{\Omega}} \right)^{n+1}
		= \int_{V} \Pi^{n+1},
\end{equation}
where $\Pi$ is the $2$-form on $V$
corresponding to $(\sqrt{-1} / 2 \pi) \Theta_{h_{\Omega}}$.
Since $K_{\Omega}|_{D}$ is isomorphic to $K_{V}|_{D}$,
the $(1, 1)$-form $\omega = \Pi|_{D}$
is a representative of $c_{1}(K_{\Omega} |_{D}) = c_{1}(K_{V}|_{D})$.
If $\Omega$ has a K\"{a}hler-Einstein metric with negative Einstein constant,
then the line bundle $K_{\Omega} |_{D} \cong K_{V}|_{D}$ is positive.

\begin{proof}[Proof of \cref{thm:computation-of-Chern-numbers}]
	Since $N_{D / \Omega}$ is a complex line bundle,
	it has a canonical $S^{1}$-action;
	its generator is denoted by $\xi$.
	Without loss of generality,
	we may assume that $\Pi$ is $S^{1}$-invariant.
	Then there exists an $S^{1}$-invariant $1$-form $\mu$ on $N_{D / \Omega}$
	such that
	\begin{equation}
		\Pi - p^{*} \omega = - d \mu,
	\end{equation}
	where $p \colon N_{D / \Omega} \to D$ is the projection.
	This is because the zero section is an $S^{1}$-equivariant deformation retract of $N_{D / \Omega}$.
	The boundary $S$ of $V$ is a principal $S^{1}$-bundle over $D$,
	and its Chern class is equal to $c_{1}(N_{D / \Omega})$.
	Near $S$, $p^{*} \omega = d \mu$ holds,
	and so
	\begin{align}
		d (\mu (\xi))
		&= \mathcal{L}_{\xi} \mu - \iota_{\xi} d \mu \\
		&= \mathcal{L}_{\xi} \mu - \iota_{\xi} (p^{*} \omega) \\
		&= 0.
	\end{align}
	Hence $\mu (\xi)$ is a constant $\alpha \in \mathbb{R}$ near $S$.
	First assume that $\alpha = 0$.
	Then there exists a $1$-form $\nu$ on $D$
	such that $\mu |_{S}= p_{S}^{*} \nu$,
	where $p_{S} \colon S \to Y$ is the canonical projection,
	and $\omega = d \nu$ holds since $p_{S}^{*} \omega = p_{S}^{*} (d \nu)$.
	Thus $\Pi = d (p^{*} \nu - \mu)$ and
	\begin{align}
		\int_{L} \Pi^{n+1}
		&= \int_{V} d (p^{*} \nu - \mu) \wedge \Pi^{n} \\
		&= \int_{S} (p^{*} \nu - \mu)|_{S} \wedge (\Pi |_{S})^{n} \\
		&= 0.
	\end{align}
	Note that $c_{1}(K_{V} |_{D}) = 0$
	and
	\begin{equation}
		c_{1}(K_{D}) = c_{1}(K_{V}|_{D}) + c_{1}(N_{D / \Omega}) = c_{1}(N_{D / \Omega})
	\end{equation}
	in $H^{2}(Y; \mathbb{R})$
	since $\omega = d \nu \in c_{1}(K_{V} |_{D})$.
	Next,
	consider the case $\alpha \neq 0$.
	Then $\alpha^{-1} \mu|_{S}$ is a connection $1$-form
	for the principal $S^{1}$-bundle $p_{S} \colon S \to D$.
	In particular,
	\begin{equation}
		c_{1} (N_{D / \Omega})
		= \left[- \frac{1}{2 \pi} d (\alpha^{-1} \mu|_{S}) \right]
		= \left[- \frac{1}{2 \pi \alpha} \omega \right].
	\end{equation}
	Therefore,
	\begin{align}
		\int_{V} \Pi^{n+1}
		&= \int_{V} (\Pi^{n+1} - (p^{*} \omega)^{n+1}) \\
		&= \int_{V} (\Pi - p^{*} \omega) \wedge (\Pi^{n}
		+ \dots + (p^{*} \omega)^{n}) \\
		&= - \int_{V} d[\mu \wedge (\Pi^{n} + \dots + (p^{*} \omega)^{n})] \\
		&= - \int_{S} \mu \wedge (p_{S}^{*} \omega)^{n} \\
		&= (- 2 \pi \alpha)^{n+1} \int_{D} c_{1} (N_{D / \Omega})^{n}.
	\end{align}
	In the last equality,
	we use the integration along fibers.
	From $\omega \in c_{1}(K_{V} |_{D})$,
	$c_{1}(K_{V} |_{D}) = (- 2 \pi \alpha) c_{1}(N_{D / \Omega})$
	and
	\begin{equation}
		c_{1}(K_{D}) = c_{1}(K_{V} |_{D}) + c_{1}(N_{D / \Omega}) = (1 - 2\pi \alpha) c_{1}(N_{D / \Omega}).
	\end{equation}
	This proves the first statement.
	
	If $K_{V} |_{D}$ is positive,
	then $\alpha$ is non-zero since $c_{1}(K_{V} |_{D}) \neq 0$.
	Hence $c_{1}(N_{D / \Omega}) = (- 2 \pi \alpha)^{-1} c_{1}(K_{V} |_{D})$,
	and $N_{D / \Omega}$ is either positive or negative.
	However,
	$N_{D / \Omega}$ cannot be positive from \cref{lem:normal-bundle-of-exceptional-set},
	and so $N_{D / \Omega}$ is negative.
\end{proof}

\section{Proof of \cref{thm:renormalized-volume-for-tubes}}

To prove \cref{thm:renormalized-volume-for-tubes},
we first show the following

\begin{lemma} \label{lem:flat-metric-on-tubes}
	The domain $\Omega$ in \cref{thm:renormalized-volume-for-tubes}
	satisfies \cref{assumption:flat-Hermitian-metric}.
\end{lemma}

\begin{proof}
	Denote by $X$ the total space of $L^{-1}$.
	It is known that the maximal compact analytic set of $\Omega$
	is the zero section of $L^{-1}$.
	Choose the whole $\Omega$ as a tubular neighborhood of the zero section.
	It is enough to prove the existence of a Hermitian metric of $K_{X}$
	such that its curvature has compact support in $\Omega$.
	Let $h_{Y}$ be the Hermitian metric of $K_{Y}$
	induced from the K\"{a}hler metric $\omega$.
	Then,
	$\sqrt{-1} \Theta_{h_{Y}} = - \beta \cdot \omega$ holds.
	Since $K_{X}$ is isomorphic to $p^{*} K_{Y} \otimes p^{*} L$,
	where $p \colon X \to Y$ is the projection,
	it has the Hermitian metric $h'_{X}$ induced from $h_{Y}$ and $h_{L}$,
	whose curvature $\sqrt{-1} \Theta_{h'_{X}}$ is equal to
	\begin{equation}
		\sqrt{-1} \Theta_{h'_{X}}
		= \sqrt{-1}(p^{*} \Theta_{h_{Y}} + p^{*} \Theta_{h_{L}})
		= \sqrt{-1}(1 - \beta) p^{*} \Theta_{h_{L}}.
	\end{equation}
	Since $\sqrt{-1} p^{*} \Theta_{h_{L}}$ has a K\"{a}hler potential outside the zero section,
	we obtain a desired Hermitian metric of $K_{X}$
	by modifying $h'_{X}$.
\end{proof}

\begin{proof}[Proof of \cref{thm:renormalized-volume-for-tubes}]
	First,
	consider the existence of a K\"{a}hler-Einstein metric on $\Omega$.
	Van Coevering~\cite[Corollary 5.6]{vanCoevering12} has proved that
	$\Omega$ has a complete K\"{a}hler-Einstein metric
	with Einstein constant $- (n + 2)$
	if and only if $c_{1}(K_{Y}) + c_{1}(L) > 0$.
	Hence if $\beta < 1$,
	such a metric $\omega_{+}$ exists.
	Moreover,
	from the proof of \cref{lem:flat-metric-on-tubes},
	$K_{X}$ has a Hermitian metric whose curvature
	has compact support in $\Omega$.
	Thus the renormalized volume of $(\Omega, \omega_{+})$
	is well-defined.
	As discussed in \cref{section:introduction},
	it is enough to compute the integral of $\widetilde{c}_{1}(K_{\Omega})^{n+1}$,
	and its formula is given by \cref{thm:computation-of-Chern-numbers}.
\end{proof}

\end{document}